\documentclass[a4paper, 11pt, english]{article}
\usepackage{preambel}

\begin{document}
\maketitle

\begin{abst}
Given a function in the Hardy space of inner harmonic gradients on the sphere, $H_+(\SS)$, we consider the problem of finding a corresponding function in the Hardy space of outer harmonic gradients on the sphere, $H_-(\SS)$, such that the sum of both functions differs from a locally supported vector field only by a tangential divergence-free contribution. We characterize the subspace of $H_+(\SS)$ that allows such a continuation and show that it is dense but not closed within $H_+(\SS)$. Furthermore, we derive the linear mapping that maps a vector field from this subspace of $H_+(\SS)$ to the corresponding unique vector field in $H_-(\SS)$. The explicit construction uses layer potentials but involves unbounded operators. We indicate some bounded extremal problems supporting a possible numerical evaluation of this mapping between the Hardy components. The original motivation to study this problem comes from an inverse magnetization problem with localization constraints.
\end{abst}

\begin{key}
	Hardy spaces, Constrained approximation, Vector field decomposition, Locally divergence-free fields, Spatially localized fields
\end{key}

\section{Introduction}\label{sec:intro}
The space of square-integrable vector fields on a sphere \( \SS \) can be written as a direct sum
\begin{align}\label{eq:hardy-hodge}
\glts = \Hp \oplus \Hm \oplus \Divfree,
\end{align}
where \( \Divfree \) is the space of divergence-free vector fields, and \( \Hp, \Hm \) are the (spherical) Hardy spaces (see below for more details).
For a generic \( \f \in \glts \) its Hardy components \( \f_{+} \in \Hp \) and \( \f_{-} \in \Hm\) are independent, implying that either one of them is not enough to reconstruct the other without knowing the entire field \( \f \).
However, if \( \f \) is localized, meaning that it vanishes on an open domain of \( \SS \), we have the following theorem.
\begin{thm}[{\cite[Cor. 2.6]{gerhards20}}]\label{thm:old}
	If  \( \f \in \glts \) is localized
	then one Hardy component of \( \f \) uniquely determines the other.
\end{thm}
This result is useful for the following type of inverse magnetization problems:
Consider a magnetic sphere and use \( \f \) to describe its magnetization.
It is known that  outside the sphere only the Hardy component
$\f_+$ generates a measurable magnetic field, whereas the Hardy component \( \f_{-} \) is invisible (e.g., \cite{gerhards16a,gubbins11}). Consequently, it is impossible to determine the magnetization of the sphere by measuring only its outer magnetic field.
Nevertheless, if the magnetization is a priori known to vanish on some open subset of \( \SS \), Theorem \ref{thm:old} implies that $\f_+$ and $\f_-$ are both determined by the magnetic field outside the sphere and thus the entire magnetization \( \f \) can be recovered (up to a divergence-free part which is never visible).
This particular setup appears in geomagnetism, where one typically models the Earth's crust as a magnetized sphere and measures its outer magnetic field on a satellite orbit (for more details
 see e.g., \cite{baratchartgerhards16,baratchart13,gerhards16a,gerhards16b}).

Devoid of the geomagnetic context, Theorem \ref{thm:old} states that the Hardy components of a localized vector field are dependent, but its proof in \cite{gerhards20} is not constructive and it does not state how to obtain one component from the other.
Furthermore, to relate the Hardy components we need to know a priori that the given \( \f_{+} \) belongs to a localized vector field $\f$. Nevertheless, Theorem \ref{thm:old} does not characterize Hardy components that appear from such localized vector fields. Thus, when looking only at a single Hardy component, we lack the means to verify whether the assumption of the theorem can be satisfied.

In this paper we close both of the above shortcomings.
In Section \ref{sec:relating}, we rephrase the problem into a more convenient setting in terms of locally divergence-free vector fields. Subsequently, we characterize the set of Hardy components that belong to strictly localized fields and provide  explicit linear mappings to reconstruct one component from the other. Nevertheless, these mappings are unbounded
 and thus difficult to evaluate numerically. For that reason, in Section \ref{sec:analofop} we suggest several bounded extremal problems to approximate the reconstruction procedure.

\paragraph{Notation.}

By \( \BB \) we will denote the open unit ball in the Euclidean space \( \R^{d} \) (\( d\geq 3 \)). Its boundary, the unit sphere, will be denoted by \( \SS \).
The space \( \lts \) will be the space of square-integrable (with respect to the surface measure) functions on \( \SS \), with the canonical scalar product \( \scalp{f, g} \) and the norm \( \| f \| = \scalp{f,f}^{\nicefrac{1}{2}} \) for \( f,g \in \lts \). For an open set \( \zeroset \) in \( \SS \) we will consider \( L^{2}(\zeroset) \) as a subset of \( \lts \) of functions with essential support in \( \zeroset \).
If \( H \) is a subspace of \( \lts \), we will write \( H / \langle 1 \rangle \) to denote the subspace of functions in \( H \) that satisfy \( \scalp{f,1} =0 \) for \( f \in H \), i.e.,  functions in \( H \) with zero mean.
If \( A \) is a linear operator on \( \lts \) then \( \dom{A} \) will denote its domain and for \( H \subseteq \dom{A} \) we will write \(A\big( H \big) =  \{ Af \colon \ f \in H \}  \) for the image of \( H \) under \( A \), and \( A\vert_{H} \) for the operator \( A \) restricted to the subspace \( H \).
The space \( \glts \) is identified with the tensor product \( \lts \otimes \R^{d} \), referring to the space of square integrable vector fields on \( \SS \). For \( \f,\g \in \glts \), we will denote their scalar product by \( \langle \f,\g \rangle \) and the norm by \( \| \f \|=\langle \f,\f \rangle^{\nicefrac{1}{2}} \).
The sphere \( \SS \) admits an outward-pointing unit normal vector field \( \n \). This field is continuous, so that for every \( \f \in \glts \) the (Euclidean) scalar product \( \n \cdot \f \) defines a function in \( \lts \) which we call the normal of \( \f \).
If \( \f \in \glts \) has a vanishing normal, we call \( \f \) a tangent vector field. If \( \f \) is a tangent vector field, we will write \( \divs(\f) \) to denote the surface divergence of \( \f \) in the distributional sense.
The space \( \Divfree \) will be the space of divergence-free vector fields, i.e., the space of vector fields \( \f \in \glts \) that have vanishing normal and whose surface divergence \( \divs(\f) \) vanishes as a distribution on the entire sphere.
If \( \zeroset \) is an open subset of \( \SS \), we will write %
\( \divs(\f)\vert_{\zeroset} \) to denote the distribution on \( \zeroset \) that results from restricting the distribution \( \divs(\f) \) to this subset (no confusion with the operator restriction shall arise since the difference will be clear from the context).
The space \( \sob \) will refer to the Sobolev space of functions \( f \) in \( \lts \) whose surface gradient \( \grads f \) is in \( \glts \); equipped with the graph norm \( \| f \|_{\sob} = \| f \| + \| \grads f \|_{\glts} \), the space \( \sob \) is itself a Hilbert space. Furthermore, if \( \zeroset \) is an open subset of \( \SS \), then \( \cc(Z) \) will be the space of smooth functions on \( \SS \) with compact support in \( \zeroset \). The closure of \( \cc(\zeroset) \) in the Sobolev norm \( \| \cdot \|_{\sob} \) is the Sobolev space \( \sobzero \) which is a closed subset of \( \sob \).

\paragraph{Layer potentials and Hardy spaces.}
For \( f \in \lts \) we write its single layer potential as
\begin{align}
    \mathcal{S}f (x)
	= \frac{-1}{\omega_d (d-2)} \
	 \int_{\SS} \frac{f(y)}{| x - y |^{d-2}}  \ d\omega(y)
	\qquad \big( x \in \R^{d} \big),
\end{align}
where \( \omega \) is the surface measure on \( \SS \), the constant $\omega_d=\omega(\SS)$ is the surface area of the sphere in $\R^d$, and \( | \cdot | \) is the Euclidean vector norm. The function \( \mathcal{S}f \) is continuous on the entire \( \R^{d} \), harmonic on \( \BB \), and harmonic on \( \R^{d} \setminus \overline{\BB} \) while vanishing at infinity.
The trace of \( \mathcal{S}f \) to the sphere, \( \trace(\mathcal{S}f) \), defines the (boundary) single layer potential $S \colon \lts \to \sob$ via
\begin{equation*}
Sf(x)
=\frac{-1}{\omega_d(d-2)}\int_{\SS} \frac{f(y)}{|x-y|^{d-2}} \ d \omega(y) = \trace(\mathcal{S}f)(x)
\qquad \big( x\in  \SS \big).
\end{equation*}
The operator \( S \) is invertible, and (on the sphere) it preserves constant functions, meaning that \( Sf \) is a constant function if and only if \( f \) is constant. By definition, every function \( f \) in the Sobolev space \( \sob \) has a surface gradient on the sphere \( \grads f \) that is in \( \glts \). Therefore, the mapping \( \grads S \) from \( \lts \) to \( \glts \) defines a bounded linear operator, whose adjoint \( (\grads S)^{\ast} \) is a bounded linear operator from \( \glts \) to \( \lts \).

\begin{rem}\label{rem:S*div}
If \( \dsob \) denotes the dual space of the Sobolev space \( \sob \), then we can define the Banach space adjoint of \( S \) as the mapping \( S^{*} \colon \dsob \to \lts \). In this case, we can identify \( (\grads S)^{\ast} \) with the mapping
\begin{align}
S^{*} \circ \divs \colon \glts \to \dsob \to \lts.
\end{align}
However, for the following analysis it will be more convenient to stay in \( \lts \) and to not switch between Sobolev spaces. For this reason we will use the operators \( (\grads S)^{\ast} \) instead of \( S^{*} \circ \divs \), even though the latter may be more familiar to some readers.
\end{rem}

The normal derivative trace \( \partial_{\n^+ }\mathcal{S}f \) and \( \partial_{\n^- }\mathcal{S}f \) to the sphere from the inside and the outside, respectively,  define the operators
\begin{align}\label{eq:double layer}
	K - \tfrac{1}{2} \id \colon \lts / \langle 1 \rangle \to \lts / \langle 1 \rangle,
	&&
	K + \tfrac{1}{2} \id \colon \lts \to \lts,
\end{align}
respectively,
where \( \id \) is the identity operator and \( K \) is the singular double layer potential
\begin{equation*}
Kf(x)=
\textnormal{p.v.}\,
\frac{-1}{\omega_d} \
\int_{\SS}
\frac{
	\n(y)\cdot (x-y)
 }{%
	|x-y|^{3}
}
f(y) \ d \omega(y) \qquad \big(x\in  \SS\big).
\end{equation*}
Both operators in \eqref{eq:double layer} are invertible and self-adjoint (on the sphere). The operator \( K + \tfrac{1}{2} \id \) preserves constant functions, while the operator \( K - \tfrac{1}{2}I \), extended to act on constant functions, annihilates them. For more information on the layer potentials we refer e.g., to  \cite{fabmenmit98,ver84} and references therein.

\paragraph{Hardy spaces and the Hardy-Hodge decomposition.}
Using the single and the double layer potentials, we now define the required Hardy spaces. Our exposition of Hardy spaces is slightly unconventional but it serves the purpose to introduce them without additional analytical tools that are necessary for the original and conventional definition (e.g., \cite{stein1960} and Appendix \ref{sec:hardyspace}). We define the Hardy spaces \( \Hp \) and \( \Hm \) as the range of the operators
\begin{equation}\label{cross_s}
	\begin{aligned}
    B_{+} \colon \lts / \langle 1 \rangle &\to \glts
	&
	B_{-} \colon \lts &\to \glts \\
	f &\mapsto \n \ \big(K-\tfrac{1}{2} \id \big)f+\grads S f
	&
	f &\mapsto \n \ \big(K+\tfrac{1}{2} I\big)f+\grads S f.
\end{aligned}
\end{equation}
Both operators in \eqref{cross_s} are bounded, have closed range in \( \glts \), and on that range they are bijective and thus (boundedly) invertible. The range of \( B_{+} \) defines the Hardy space \( \Hp \) of non-tangential limits of harmonic gradients from within the ball \( \BB \) to the sphere, while the range of \( B_{-} \) defines the Hardy space \( \Hm \) of non-tangential limits of harmonic gradients from outside the ball \( \BB \) to the sphere.

The Hardy-Hodge decomposition of the space \( \glts \) is the orthogonal direct sum \eqref{eq:hardy-hodge}.
Thus, every vector field \( \f \in \glts \) can be uniquely written as
\(    \f = B_{+}\varphi + B_{-} \phi + \f_{df}
\) for some unique \( \varphi, \phi \in \lts \) with \( \scalp{\varphi, 1} = 0 \). For brevity, we will sometimes write \( \f_{+} \) for \( B_{+}\varphi \) and \( \f_{-} \) for \( B_{-}\phi \), and refer to \( \f_{+} \) and \( \f_{-} \) as to the Hardy components of \( \f \).
Since the Hardy-Hodge decomposition is orthogonal, it follows that for every \( \varphi, \phi \in \lts \) we have
\begin{align}
    \scalp{B_{+} \phi, B_{-}\varphi}_{\glts} =0.
\end{align}
By inserting the definition of the operators \( B_{+} \) and \( B_{-} \) in this equation and using the adjoint of \( \grads S \) we get the operator identity
\begin{align}\label{eq:single-double layer}
    (\grads S)^{\ast} \grads S
	= - \big( K + \tfrac{1}{2}I \big) \big( K - \tfrac{1}{2}I \big).
\end{align}
We will use this relation several times in the following analysis. It is worth to point out that this relation holds only on a sphere, since on other surfaces the Hardy-spaces (if defined) are not orthogonal.
Further details regarding Hardy spaces and the Hardy-Hodge decomposition can be found, e.g., in \cite{atfeh10,bpt20,bargerkeg20,baratchart13,gerhards16a}, which also consider more general surfaces than the sphere.

\subsection{Localized and locally divergence-free vector fields}

Theorem \ref{thm:old} discusses localized fields -- fields that vanish on an open set of \( \SS \). For the further analysis, however, it appears convenient to rephrase this condition solely in terms of Hardy components of a field. We observe that writing \( \f \in \glts \) in the Hardy-Hodge decomposition and assuming that $\f$ vanishes on the open set \( \zeroset\) yields \( \f_{+} + \f_{-} = - \f_{df} \) on \( \zeroset \). In other words, if a vector field vanishes on an open domain of a sphere, then on this domain the sum of its Hardy components must have a vanishing normal and it must be divergence-free (in the distributional sense). We call such fields locally divergence-free on \( \zeroset \).

The converse statement holds in the following form: if \( \f \) is a locally divergence-free field on \( \zeroset \), then for every open set \( O \subset \zeroset \) with a sufficiently regular boundary in the sense the boundary separate \( \SS \) into two Lipschitz domains, there exists a \( \g \in \Divfree \) that is divergence-free  on the entire sphere  such that \( \f + \g \) vanishes on \( O \). This statement follows from the fact that if the boundary of \( O \) satisfies the above conditions, then we can always extend a locally divergence-free field to a divergence-free field on the entire sphere (for the convenience of the reader we include the necessary extension result and its proof in the Appendix \ref{sec:cont}).
Since \( \Divfree \) is orthogonal to the Hardy spaces, adding an appropriate \( \g \in \Divfree\) to \( \f \) leaves the Hardy components of \( \f \) unchanged. Thus, looking only at Hardy components of a vector field, it makes no difference whether we assume the field \( \f \) to vanish on an open set or to be locally divergence-free. We summarize the above discussion in the following proposition.

\begin{prop}\label{prop: equivalence}
Let \( \zeroset \subset \SS \) be a Lipschitz domain with connected boundary. If \( \f \in \glts \) vanishes on \( \zeroset \), then the sum of its Hardy-components \( \f_{+} + \f_{-} \) is locally divergence-free on \( \zeroset \). Conversely, if for given \( \f_{+} \in \Hp \) and \( \f_{-} \in \Hm \) the sum \( \f_{+} + \f_{-} \) is locally divergence-free on \( \zeroset \), then there exists a field \( \f \in \glts \) that vanishes on \( \zeroset \) and has Hardy-components \( \f_{+} \) and \( \f_{-} \).
\end{prop}

If \( \zeroset \) is open in \( \SS \), we can always choose a subset \( O \) of \( \zeroset \)  with a sufficiently regular boundary to satisfy the requirement of the proposition above. Thus, with the aid of  Proposition \ref{prop: equivalence} our initial goal of finding a relation between the Hardy components of strictly localized vector fields becomes equivalent to finding a relation between Hardy components \( \f_{+} \) and \( \f_{-} \) under the assumption that their sum is locally divergence-free on \( \zeroset \). This is what we investigate in the remaining part of the paper.

\section{Hardy components of locally divergence-free vector fields}\label{sec:relating}

Our goal of this section is to prove Theorem \ref{thm:main}, which defines an ``if and only if'' condition for a vector field to be locally divergence-free on an open subset of a sphere.
To this aim, the following space will be pivotal:
 for an open set \( \zeroset \) in \( \SS \) we define
\begin{align}
    \divfreezeroset
	= \big\{
    f \in \lts \colon \divs (\grads S f)\vert_{\zeroset} = 0
    \text{ as distribution on } \zeroset
    \big\}.
\end{align}
This space is non-empty as it contains constant functions: if \( c \) is a constant function in \( \lts \), then \( Sc \) is again constant and \( \grads S c = 0 \). Moreover, \( \divfreezeroset \) is closed in \( \lts \). To see this consider a Cauchy sequence  \( ( f_{n})_{n\in\bbN}  \) in \( \divfreezeroset \) that converges to \( f \in \lts \). Then for
  \( \varphi \in \cc(\zeroset)\), it holds
\begin{align}
	\label{eq: closed space}
    \big| \divs(\grads Sf) (\varphi) \big|
    = \big| \scalp{\grads Sf , \grads \varphi} \big|
	= \big| \scalp{\grads S(f-f_{n}) , \grads \varphi} \big|
    \leq C \big\| f - f_{n} \big\|,
\end{align}
where \( C \) is a finite constant that depends on \( \varphi \). The right-hand side of
\eqref{eq: closed space} goes to zero as \( n \) grows to infinity,
implying that \( f \) is in \( \divfreezeroset \) and that \( \divfreezeroset \) is a closed  subspace of \( \lts \).

The space \( \divfreezeroset \) defines functions
that we call \textit{\( S \)-harmonic} on \( \zeroset \).
If \( f \) is in \( \lts \), then its surface gradient is naturally a distribution and thus not necessarily square integrable.
The function \( S f \), on the other hand, is in the Sobolev space \( \sob \), and \( \grads S f \) is always a square integrable field in \( \glts \).
Moreover, since \( S \) is one-to-one and onto \( \sob \), we can unambiguously define \( \grads S f \) as an \textit{\( L^2 \)-surface gradient} of \( f \). Thus, \( \divfreezeroset \) defines functions with a divergence-free \( L^{2} \)-surface gradient on \( \zeroset \). Consequently, if \( f \) is in \( \divfreezeroset \) then vanishing divergence of \(\nabla Sf \) implies that \( Sf \) is weakly harmonic on \( \zeroset \). By Weyl's Lemma \( Sf \) is harmonic on \( \zeroset \), which we summarize by saying that \( f \) is \( S \)-harmonic there.

The orthogonal complement of \( \divfreezeroset \) will also be important in what follows. It is characterized by the following lemma that we prove in the Appendix \ref{sec:prflemma}.

\begin{lem}\label{lem:orthogonal complement}
Let \( \zeroset \subset \SS \) be open and \( \SS \setminus \zeroset \) have positive surface measure, then the orthogonal complement of \( \divfreezeroset \) is the set
\begin{align}\label{eq:orthogonal complement}
    \divfreezeroset^{\perp}
	=
		\big\{
			\big( K + \tfrac{1}{2} I \big) \big( K - \tfrac{1}{2}I \big)
			(S^{-1}\varphi) \colon \varphi \in \sobzero
		\big\}.
\end{align}
\end{lem}
In the above Lemma we require the mild condition that \( \SS \setminus \zeroset \) must have positive measure. However, in the following it will be necessary to strengthen this requirement and to demand that \( \SS \setminus \zeroset \) contains an open set. To ensure this, we will write \( \zeroset \Subset \SS \), meaning that \( \zeroset \) and its closure \( \overline{\zeroset} \) are both proper subsets of \( \SS \). The following theorem indicates why this additional condition on the closure of \( \zeroset \) is important.
Moreover, for the rest of the paper we will use the following notation: if \( \zeroset \) is an open subset of \( \SS \), then \( \proj_\zeroset  \) will denote the orthogonal projection from \( L^{2}(\SS) \) onto \( L^{2}(\zeroset) \). If the set \( \zeroset \) is clear from the context, we will write \( \proj \) instead of \( \proj_\zeroset  \) and call it \textit{the corresponding projection} to the set \( \zeroset \).

\begin{thm}\label{prop:dense1}
	Let \( \zeroset \subseteq \SS \) be open and \( \proj \) be the corresponding projection in \( \lts \).
The operators \( \proj \left(K - \frac{1}{2} \id \right):\divfreezeroset /\langle1\rangle \to L^2(\zeroset) \) and \( \proj \left(K + \frac{1}{2} \id \right):\divfreezeroset \to L^2(\zeroset) \) are injective. If additionally \( \overline{\zeroset} \) is a proper subset of \( \SS \), then the range of both operators is dense in \( L^{2}(\zeroset) \), but neither range is closed.
\end{thm}

\begin{proof} %
For injectivity of \( \proj (K- \tfrac{1}{2} \id) \), assume a \( g \in \divfreezeroset / \langle 1\rangle \) such that
\begin{equation}\label{Eq1}
\proj \left(K-\tfrac{1}{2} \id\right)g=0.
\end{equation}
For an open cone \( \mathcal{C}(\zeroset)=\{r\xi: \xi\in \zeroset, \ r \in (0,\infty)\} \) set \( \mathcal{C}_{\BB}(\zeroset) = \mathcal{C}(\zeroset) \cap \BB \) and define two functions \( u \) and \( w \)  as
\begin{align}\label{eqn:DNbv}
	u(r\xi) \coloneqq \mathcal{S}g(r\xi), &&
    w (r \xi) \coloneqq Sg (\xi) \qquad \big(\xi \in \zeroset, \ r \in \R, \ r\xi \in {\mathcal{C}_{\BB}(\zeroset)} \big).
\end{align}
Observe that by the definition of layer potentials in \eqref{eq:double layer}, condition \eqref{Eq1} states that the normal derivative of \( u \) from the inside to the sphere, \( \partial_{\n^+} u \), vanishes on \( \zeroset \).
Since \( g \) is in \( \divfreezeroset \), it is \( S \)-harmonic.
Consequently, \( \Delta w(r\xi)= r^{-2}\laps Sg (\xi)=0 \), and the function \( w \) is harmonic in \( \mathcal{C}_{\BB} (\zeroset)\). It follows that the functions \( u \) and \( w \) are both harmonic in \( \mathcal{C}_{\BB}(\zeroset) \), coincide on \( \zeroset \), and have a vanishing normal derivative on \( \zeroset \).
Set \( v = u - w \). Then \( v \) is continuous on \( \mathcal{C}_{\BB}(\zeroset) \cup \zeroset \), vanishes on \( \zeroset \), and its normal derivative also vanishes on \( \zeroset \).
Therefore, extending \( v \) by zero to the whole cone \( \mathcal{C}(\zeroset) \) defines a harmonic function. This function vanishes on an open set \( \mathcal{C}(\zeroset) \cap (\R^{d}\setminus \overline{\BB}) \) and coincides with \( v \) on \( \mathcal{C}_{\BB}(\zeroset) \). By analyticity of harmonic functions,  \( v \) must be identically zero,
which yields \( u(r \xi) =w(r\xi)=Sg(\xi) \) for all \( \xi\in \zeroset \) and \(r \in (0,1)\). By continuity of \( u \) at the origin, we have
\begin{align}
    Sg(\xi)=\lim_{r\rightarrow 0}u(r\xi)= u(0) \qquad \big( \xi\in\zeroset \big),
\end{align}
and hence, $u$ must be constant in $\mathcal{C}_{\BB}(\zeroset)$.
Because $u$ coincides with \( \mathcal{S}g \) on the non-empty open set $\Sigma$ and \( \mathcal{S}g \) is harmonic in all of $\BB$, it follows that \( \mathcal{S}g \) must be a constant function on $\BB$. Taking the trace to the sphere yields that \( Sg \) is constant on $\SS$. Since the single layer potential is injective and on the sphere preserves constant functions, \( g \) must be constant, and since the only constant function in \( \divfreezeroset / \langle 1 \rangle \) is the zero function, this proves injectivity.

The injectivity of \( \proj \left( K + \frac{1}{2} \id \right) \) follows the same line of reasoning as above,  after exchanging the role of \( \mathcal{C}_{\BB}(\zeroset) \) and \( \mathcal{C}(\zeroset) \cap (\R^{d} \setminus \overline{\BB}) \), and replacing the limit \( r \to 0 \) by  \( r \to \infty \). Since the region \( \mathcal{C}(\zeroset) \cap (\R^{d} \setminus \overline{\BB} )\) is unbounded, the preceding argument implies that \( \mathcal{S}g \) is not only constant but indeed zero outside the ball \( \BB \), and consequently \( Sg \) must vanish on \( \SS \). Since \( S \) is injective it follows that \( g \) must be identically zero.

If \( \omega \) is an open subset of \( \zeroset \) and \( \proj_{\omega} \) is the corresponding projection then the above argument implies that \( \proj_{\omega} \left( K \pm \tfrac{1}{2} \id \right) g \) is injective. Thus, the functions \( \left( K \pm \tfrac{1}{2} \id \right) g \) cannot vanish on any open subset of \( \zeroset \)  unless being identically zero. Hence, the range of \( \proj \left( K \pm \tfrac{1}{2} \id \right) \) is not the whole space \( L^{2}(\zeroset) \).

To prove the density of the range, let \( f \in L^{2}(\zeroset) \) be such that
\begin{align}\label{eq:locality condition}
    0
	= \scalp{f, \proj \left(K - \tfrac{1}{2} \id \right) g}
	= \scalp{ \left(K - \tfrac{1}{2} \id \right) f, g}
	\qquad \big(  g \in \divfreezeroset \slash \langle 1 \rangle \big),
\end{align}
where the second equality holds because \( K \) is self-adjoint and \( \proj f = f \). Thus, \( (K - \tfrac{1}{2}I) f  \) is in \( \divfreezeroset^{\perp} \) (since constant functions are not in the range of \( K - \tfrac{1}{2}I \)). From \eqref{eq:orthogonal complement}, there exists a function \( \varphi \in \sobzero \) such that
\(
    ( K-\tfrac{1}{2}I )\big(( K + \tfrac{1}{2}I ) S^{-1} \varphi - f \big) = 0 \).
Since \( \big( K - \tfrac{1}{2}I \big) \) is invertible on \( \lts \slash \scalp{1} \),  there must exists a constant \( c \) such that \( (K + \tfrac{1}{2}I) S^{-1} \varphi + c = f\). Since \( K+ \tfrac{1}{2}I \) and \( S \) preserve constant functions, there is a (possibly different) constant \( c \) with
\begin{align}\label{eq:density 1}
    \big( K + \tfrac{1}{2}I \big) S^{-1} (\varphi + c) = f.
\end{align}
If \( \overline{\zeroset} \) is a proper subset of \( \SS \), then \( \suppoff = \SS \setminus \overline{\zeroset} \) is open and the corresponding projection onto $L^2(\suppoff)$ is \( \id - \proj \).
Now, the function \( S^{-1} (\varphi + c) \) is in the set \( \divfreesuppoff \) since \( \grads S (S^{-1}(\varphi+ c)) = \grads (\varphi + c) = \grads \varphi \) vanishes on \( \suppoff \).
Moreover, since \( f \) also vanishes on \( \suppoff \), we have
\begin{align}\label{eq:surjectivity}
    (I - \proj) \big( K + \tfrac{1}{2}I \big) S^{-1} (\varphi + c) = (I - \proj)f = 0.
\end{align}
From the first part of the proof we know that \( (I - \proj) (K + \tfrac{1}{2}I) \) is injective on \( \divfreesuppoff \) and thus \eqref{eq:surjectivity} implies that \( S^{-1}(\varphi + c) \) must be zero. From \eqref{eq:density 1} we therefore get \( f = 0 \). This shows that the orthogonal complement of the range of \( \proj (K - \tfrac{1}{2} \id) \) contains only the zero function, i.e., the range is dense in \( L^{2}(\zeroset) \). At the same time, the range cannot be closed since it is not the whole \( L^{2}(\zeroset) \) (as we have noticed in the paragraph after proving injectivity). This proves the claim for the operator \( \proj \big( K-\tfrac{1}{2} \id \big) \).

To show density for the operator \( \proj (K + \tfrac{1}{2}\id) \) we follow similar arguments. Assume an \( f \in L^{2}(\zeroset) \) such that
\begin{align}\label{eq:locality condition 2}
    0
	= \scalp{f, \proj \left(K + \tfrac{1}{2} \id \right) g}
	= \scalp{ \left(K + \tfrac{1}{2} \id \right) f, g}
	\qquad \big(  g \in \divfreezeroset \big).
\end{align}
Then \( (K+ \tfrac{1}{2}I)f \) must be in \( \divfreezeroset^{\perp} \). Again, from \eqref{eq:orthogonal complement} and the fact that \( K + \tfrac{1}{2}I \) is invertible on the entire \( \lts \), there must be a \( \varphi \in \sobzero \) such that
\begin{align}\label{eq:density 2}
    \big( K - \tfrac{1}{2}I \big) S^{-1}\varphi = f.
\end{align}
With the same argument as above, the function \( S^{-1}\varphi \) is in \( \divfreesuppoff \) and \( (I-P)f = 0 \). By injectivity of the operator \( (I - \proj)(K - \tfrac{1}{2}I) \) on \( \divfreesuppoff \slash \scalp{1} \), the function \( S^{-1}\varphi \) must be constant and \( f = 0 \) from \eqref{eq:density 2} (since \( K - \tfrac{1}{2}I \) annihilates constant functions). Consequently, the range of \( P(K + \tfrac{1}{2}I) \) is dense in \( L^{2}(\zeroset) \) but,  as already stated, it is not the whole \( L^{2}(\zeroset) \).
\end{proof}

To avoid pathological cases, we will henceforth assume \( \zeroset \Subset \SS \). Then, the above theorem asserts that the operators \(  \proj \left( K + \tfrac{1}{2} \id \right) \) and \(  \proj \left( K - \tfrac{1}{2} \id \right) \) (acting on their corresponding domains) are invertible on their dense ranges, but the inverses are unbounded.
This observation leads us to define the following spaces: for \( \zeroset \Subset \SS \) open, we define
\begin{align}
    \uniqueplus  = \big\{ f \in \lts \colon \ \proj_\zeroset f = \proj_\zeroset(K + \tfrac{1}{2} \id) g, \text{ for some } g \in \divfreezeroset \big\} \slash \scalp{1},\label{eqn:dps}
\end{align}
and
\begin{align}
    \uniqueminus  = \big\{ f \in \lts \colon \proj_\zeroset f = \proj_\zeroset (K - \tfrac{1}{2} \id) g, \text{ for some } g \in \divfreezeroset / \langle 1 \rangle \big\}.\label{eqn:dms}
\end{align}
By the injectivity from Theorem \ref{prop:dense1}, if a \( g \) in the above definition exists, it must be unique;
by the density assertion of the same theorem, the space \( \uniqueplus  \) is dense in \( \lts \slash \scalp{1} \) and the space \( \uniqueminus  \) is dense in \( \lts \). We use these spaces to define the following operators.

\begin{defi}\label{def:ptm}
For \( \zeroset \Subset \SS \) open, we define the operator
\begin{align}
	\begin{split}
    \ptm \colon \lts/\langle1\rangle &\to \lts \\
	f &\mapsto \left[ \proj (K + \tfrac{1}{2}\id) \right]^{-1} \proj f - f,
	\end{split}
\end{align}
acting on the dense domain \( \dom{\ptm} = \uniqueplus  \),
and the operator
\begin{align}
	\begin{split}
	\mtp \colon \lts &\to \lts/\langle1\rangle \\
	f &\mapsto -\left[ \proj (K - \tfrac{1}{2}\id) \right]^{-1} \proj f - f + \scalp{f, 1},
	\end{split}
\end{align}
acting on the dense domain \( \dom{\mtp}= \uniqueminus  \).
\end{defi}

Both operators are unbounded, as the operators \( \proj (K \pm \tfrac{1}{2}) \) have no closed range.
\begin{rem}
For better readability, we mention that the indices on the operators $\ptm$ and $\mtp$ stand for  ``plus to minus'' or ``minus to plus'', which will indicate the mappings between Hardy components.
\end{rem}

We have the following relation between the domains of \( \ptm \) and \( \mtp \).

\begin{lem}\label{lem: domains of ptm}
	For \( \zeroset \Subset \SS \) open, we have  \( \ptm \big( \uniqueplus  \big) \subseteq \uniqueminus  \) and \( \mtp \big( \uniqueminus  \big) \subseteq \uniqueplus  \).
\end{lem}

\begin{proof}
Let \( \varphi \) be in \( \uniqueplus  \), and set
\begin{align}\label{eq:phi from varphi 1}
    \phi = \ptm(\varphi)
	= \Big[\proj \big( K + \tfrac{1}{2} \id  \big) \Big]^{-1} \proj \varphi - \varphi .
\end{align}
From the definition of the space \( \uniqueplus  \), we have \( \proj \varphi = \proj (K + \tfrac{1}{2} \id) g \) for some \( g \in \divfreezeroset \). Using this fact in \eqref{eq:phi from varphi 1} and applying \( \proj \) on both sides of the equation, we get
\begin{align}
  \proj \phi
  = \proj g - \proj (K + \tfrac{1}{2} \id) g
  = -\proj (K - \tfrac{1}{2} \id) g
  = -\proj (K - \tfrac{1}{2} \id) (g - \scalp{g,1}),
\end{align}
where the last equality follows since \( \big( K - \tfrac{1}{2} \id \big) \) annihilates constant functions. Since \( g \) and \( \scalp{g,1} \) (viewed as a constant function) are both in \( \divfreezeroset \), we have \( g- \scalp{g,1} \in \divfreezeroset / \langle 1 \rangle \) and consequently
 \( \phi \in \uniqueminus \).

For the second inclusion, let \( \phi \) be in \( \uniqueminus  \) and set
\begin{align}\label{eq:map of mtp}
    \varphi = \mtp(\phi) = -\Big[ \proj \big( K - \tfrac{1}{2}I \big) \Big]^{-1}  \proj \phi - \phi + \scalp{\phi,1}.
\end{align}
From the definition of \( \uniqueminus  \), we have \( \proj \phi = \proj(K - \tfrac{1}{2}I)g \) for some \( g \in \divfreezeroset \slash \scalp{1} \). Using this fact in \eqref{eq:map of mtp}, we get \( \scalp{\varphi,1} = -\scalp{g,1} = 0 \), and moreover,
\begin{align}
    \proj \varphi = -\proj g - \proj (K - \tfrac{1}{2}I)g + \proj \scalp{\phi,1}
	= - \proj(K + \tfrac{1}{2}I)g + \proj \scalp{\phi,1}.
\end{align}
Since \( K + \tfrac{1}{2}I \) preserves constant functions, the above equation implies \(\proj \varphi = - \proj(K + \tfrac{1}{2}I) (g + c)  \) for some constant function \( c \). Moreover, since constant functions are in \( \divfreezeroset \) and \( g \) is in \( \divfreezeroset \slash \scalp{1} \), the function \( g+c \) is in \( \divfreezeroset \).
Therefore, \( \varphi \) is in \( \uniqueplus  \).
\end{proof}

We now state the main theorem of this section.

\begin{thm}\label{thm:main}
Let \( \zeroset \Subset \SS \) be open.
If \( \f \in \glts \) is tangent on \( \zeroset \) and \( \divs(\f)\vert_{\zeroset} = 0 \),
then its Hardy component \( \f_{+} \) is in \( B_{+}\big(\uniqueplus \big) \) and its Hardy component \( \f_{-} \) is in \( B_{-}\big(\uniqueminus \big) \).
 Conversely, for every \( \f_{+} \in B_{+}\big(\uniqueplus \big) \) there is a unique \( \f_{-} \in B_{-} \big(\uniqueminus  \big)\) such that \( \f_{+} + \f_{-} \) is tangent and divergence-free on \( \zeroset \). The last statement also holds with \( \f_{+} \) and \( \f_{-} \) interchanged.
\end{thm}

\begin{proof}
	Using the operators from \eqref{cross_s} we can express the Hardy-Hodge components of \( \f \) as \( \f_{+} = B_{+} \varphi \) and \( \f_{-} = B_{-} \phi \) for some unique \( \varphi, \phi \in \lts \) with \( \scalp{\varphi, 1} = 0 \).

For the first assertion, set \( \proj = \proj_{\zeroset} \) and  assume that \( \f \) is tangent and divergence-free on \( \zeroset \). Using \eqref{cross_s}, the normal and the tangent components of $\f$ restricted to \( \zeroset \) read
\begin{align}\label{eq:vanishing components}
    \proj \big(K-\tfrac{1}{2} \id \big) \varphi +\proj \big(K + \tfrac{1}{2} \id \big) \phi = 0, &&
	\divs\big(\grads S (\varphi + \phi)\big)\vert_{\zeroset} = 0.
\end{align}
From the second equation in \eqref{eq:vanishing components} it follows that the function \( \varphi + \phi \) is in \( \divfreezeroset \). Writing \( \phi = (\varphi + \phi) - \varphi \), inserting it into the first equation in \eqref{eq:vanishing components} and rearranging the terms yields
\begin{align}\label{eq:plus determines minus}
    \proj \varphi =
	 \proj \big( K + \tfrac{1}{2} \id \big)(\varphi + \phi).
\end{align}
Hence \( \varphi \) is in \( \uniqueplus  \) and \( \f_{+} = B_{+} \varphi \) is in \( B_{+}(\uniqueplus ) \).
To see that also the \( \f_{-} \) component is in the desired space, let \( c = \scalp{\phi, 1} \) be the mean of \( \phi \). Then, writing \( \varphi = (\varphi + \phi - c) - (\phi -c) \), performing the same manipulations as above, and using the fact that \( \big(K - \tfrac{1}{2}\id \big)c = 0 \), yields
\begin{align}\label{eq:minus determines plus}
    \proj \phi = - \proj \big( K - \tfrac{1}{2} \id \big) (\varphi + \phi - c).
\end{align}
The function \( \varphi + \phi \) and the constant function \( c \) are both in \( \divfreezeroset \). And since \( \varphi + \phi - c \) has zero mean, it is in \( \divfreezeroset / \langle 1\rangle \),
implying that \( \phi \) is in \( \uniqueminus  \) and \( \f_{-} = B_{-} \phi \) is in \( B_{-}(\uniqueminus ) \). Clearly, both components are unique by the direct-sum property of the Hardy-Hodge decomposition.

For the converse statement,  assume an \( \f_{+} = B_{+} \varphi \) with \( \varphi \in \uniqueplus  \) and set
\begin{align}\label{eq:phi from varphi}
    \phi = \ptm(\varphi)
	= \Big[\proj \big( K + \tfrac{1}{2} \id  \big) \Big]^{-1} \proj \varphi - \varphi.
\end{align}
By Lemma \ref{lem: domains of ptm} we have \( \phi = \ptm(\varphi) \in \uniqueminus  \),
and choosing \( \f_{-} = B_{-} \phi\), we get from \eqref{eq:phi from varphi} that
\begin{align}\label{eq:f++f-}
     \f_{+} + \f_{-} &= B_{+} \varphi + B_{-} \phi \nonumber\\
	 &= \n \, (\id - \proj) \Big[ \big(K - \tfrac{1}{2} \id \big) \varphi + \big(K + \tfrac{1}{2} \id \big) \phi \Big]
	 + \grads S (\varphi + \phi).
\end{align}
Thus, the normal component of \( \f_{+} + \f_{-}  \) vanishes on \( \zeroset \). Furthermore, it holds that \( \varphi + \phi \) is in \( \divfreezeroset \) because \eqref{eq:phi from varphi} and the definition of $\uniqueplus$ imply the existence of a $g$ in \( \divfreezeroset \) with
\begin{align}\label{eqn:g}
     \varphi+\phi=\Big[\proj \big( K + \tfrac{1}{2} \id  \big) \Big]^{-1} \proj \varphi=g.
\end{align}
Eventually, \eqref{eq:f++f-} implies that \( \f_{+} + \f_{-} \) is tangent and divergence-free on \( \zeroset \) as asserted.

Let \( \phi \) be as in \eqref{eq:phi from varphi} and assume that there is another \( \tilde{\phi} \in \lts \) leading to the field \( \tilde{\f}_{-} = B_{-}\tilde{\phi} \) such that \( \f_{+} + \tilde{\f}_{-} \) is tangent and divergence-free on \( \zeroset \). Then, from the first part of the proof, \eqref{eq:plus determines minus} holds for both, \( \phi \) and \( \tilde{\phi} \), and thus
\begin{align}
    \proj (K + \tfrac{1}{2} \id) (\varphi + \phi)
	= \proj \varphi = \proj (K + \tfrac{1}{2} \id) (\varphi + \tilde{\phi}).
\end{align}
By injectivity of \( \proj (K+ \tfrac{1}{2} \id) \) from Theorem \ref{prop:dense1}, it follows that \( \tilde{\phi} = \phi \).

Using the same arguments with obvious modifications, we see that for every \( \f_{-} = B_{-} \phi \) with \( \phi \in \uniqueminus  \), the field \( \f_{+} = B_{+} \mtp(\phi) \) satisfies the stated conditions and is also unique.
\end{proof}

\begin{cor}\label{cor:subspacedense}
For \( \zeroset \Subset \SS\) open, the space \( B_{+}(\uniqueplus ) \) is dense in \(\Hp\) and the space \( B_{-}(\uniqueminus ) \) is dense in \( \Hm \). %
\end{cor}

\begin{proof}
Since \( \uniqueplus  \) is dense in \( \lts/\langle1\rangle \) and the operator \( B_{+} \colon \lts/\langle1\rangle \to \Hp\) is bounded and invertible, the image \( B_{+}\big( \uniqueplus  \big) \) is dense in \( B_{+}\big(\lts/\langle1\rangle \big) = \Hp \). Similar for \( \Hm \).
\end{proof}

\begin{cor}\label{cor:map between hardy}
For \( \zeroset \Subset \SS \) open, \( \varphi \in \uniqueplus  \), and \( \phi \in \uniqueminus  \), the functions \(  \ptm(\varphi) \) (resp. \( \mtp(\phi) \))is the unique function in \( \lts \) (resp. \( \lts\slash \scalp{1}\))such that the fields
\begin{align}
    \f = B_{+}\varphi + B_{-}\ptm(\varphi)
	&&
	\text{and}
	&&
	\g = B_{+} \mtp(\phi) + B_{-}\phi
\end{align}
are tangent and divergence-free on \( \zeroset \).

\begin{proof}
This is shown in the calculation after \eqref{eq:phi from varphi} and in the last sentence of that proof.
\end{proof}
\end{cor}

The mappings stated in Corollary  \ref{cor:map between hardy} are the desired mappings between the Hardy components of a locally divergence-free vector field. Moreover, for a given Hardy field \( \f_{+} = B_{+} \varphi \) we only need to check whether or not \( \varphi \) is in \( \uniqueplus  \) to know if it belongs to a locally divergence-free field. This closes both of the shortcomings of Theorem \ref{thm:old}. Yet, the operators \( \ptm \) and \( \mtp \) are unbounded rendering the reconstruction of one Hardy component from the other unstable.
In particular, the operators are difficult to evaluate numerically. Therefore, in the following section, we will introduce two variational approaches that are better suited for numerical applications.

\begin{rem}
The structure of the domains $\uniqueplus$ and $\uniqueminus$ from \eqref{eqn:dps} and \eqref{eqn:dms}, respectively, potentially allows for the generation of a suitable system of basis function if one has systems of basis functions for $\divfreezeroset$ and $L^2(\SS\setminus\zeroset)$, respectively. Various approaches to the latter can be found, e.g., in \cite{freedengerhards12,freeden98,schreiner97,simons06,wendland05} and references therein. Having a system of basis functions for $\uniqueplus$ could help to numerically evaluate the bounded extremal problem \eqref{eqn:BEP1} as introduced in the upcoming section. However, such numerical aspects are not part of the paper at hand.
\end{rem}

\section{Bounded extremal problems for Hardy components}\label{sec:analofop}

In the previous section we showed that the map between Hardy components is unbounded and thus not continuous.
 The question is then, under what conditions and in what sense can we compute $\ptm(\varphi)$ (resp. $\mtp(\phi)$)  taking into account that the input data \( \varphi \) (resp. \( \phi \)) are typically corrupted by noise.
In this section, we address this question by analyzing some related bounded extremal problems. We focus on the operator $\ptm$ only, but an analogous analysis can be applied to $\mtp$ without significant difference (noting that the operator $\ptm$ is the one most important, e.g., for applications to planetary magnetic fields, since measurements are typically only available outside the planets).
Opposed to many other inverse problems, our case is special as both operators $\ptm$ and its inverse are unbounded.

\subsection{Properties of $\ptm$, $\mtp$, and the first bounded extremal problem}\label{sec:bep1}

Given the input \( \varphi \), numerical calculations of \( \ptm(\varphi) \) can be challenging. Then, it is sometimes convenient to use an \textit{approximation sequence} of functions \( \varphi_{n} \) that converge to \( \varphi \) and for which the functions \( \ptm(\varphi_{n}) \) enjoy desirable numerical properties.
However, since $\ptm$ is unbounded, %
convergence of $\varphi_n$ to  $\varphi$ does not guarantee that the sequence $(\ptm(\varphi_n))$ approaches $\ptm(\varphi)$, nor that it converges at all. To avoid this issue, we consider the following bounded extremal problem that stems form the theory of constrained approximations in Hilbert spaces.

\paragraph{Bounded Extremal Problem 1.} \label{BEP1}Let $f\in \lts$ and $\Cone>0$ be fixed. Find a $\varphi$ in $\uniqueplus$ with $\|\ptm(\varphi)\|\leq \Cone$ such that
\begin{equation}\label{eqn:BEP1}\tag{$BEP_{1}$}
	\begin{aligned}
		 \| \varphi - f\|=\inf\left\{\| g- f\|:\,g\in \uniqueplus,\,\|\ptm(g)\|\leq \Cone\right\}.
	\end{aligned}
\end{equation}

In order to address the problem above, in particular the existence of its solution, we first require some additional properties of the operators $\ptm$ and $\mtp$.

\begin{prop}\label{prop:analysis of ptm}
	Let \( \zeroset \Subset \SS \) be open and denote the identity operator on \( \lts \) by \( I \). Then \( \mtp \) is the unbounded two-sided inverse of \( \ptm \) in the following sense
	\begin{enumerate}

		\item\label{prop it:1}
		we have \( \mtp \circ \ptm=I\vert_{\uniqueplus } \) and \( \uniqueplus  = \mtp \big( \uniqueminus  \big) \);

		\item\label{prop it:2}
		we have \( \ptm \circ \mtp=I\vert_{\uniqueminus } \) and \( \uniqueminus  = \ptm \big( \uniqueplus  \big) \).

	\end{enumerate}
\end{prop}

\begin{proof}
	For \ref{prop it:1}: if \( \varphi \) is in \( \uniqueplus  \), then \( \ptm( \varphi ) \) is in \( \uniqueminus  \) by Lemma \ref{lem: domains of ptm}. Thus, \( \mtp \circ \ptm \) is well-defined on \( \uniqueplus \). Consider the mapping sequence \( \varphi \mapsto \ptm(\varphi) \mapsto \mtp \circ \ptm  (\varphi) \), and set for brevity \( \phi = \ptm(\varphi) \), \( \tilde{\varphi} =\mtp \circ \ptm (\varphi) \). Then, by Corollary \ref{cor:map between hardy}, \( B_{+}\varphi \) and \( B_{+}\tilde{\varphi} \) are unique fields such that both \( B_{+}\varphi + B_{-}\phi \) and \( B_{+}\tilde{\varphi} + B_{-}\phi \) are locally divergence-free.
	From the uniqueness assertion, it follows that \( B_{+}(\varphi - \tilde{\varphi})=0 \). Since \( B_{+} \) is injective on \( \lts \slash \scalp{1} \) and since functions in \( \uniqueplus  \) have zero mean, this yields \( \tilde{\varphi} = \varphi \) and \( \varphi = \mtp(\phi) =\mtp \circ \ptm (\varphi) \). This shows that if \( \varphi \) is in \( \uniqueplus \) then \( \varphi = \ptm (\phi) \) for some \( \phi \in \uniqueminus  \) and thus, \( \uniqueplus  \subseteq \mtp \big( \uniqueminus  \big) \). Lemma \ref{lem: domains of ptm} provides the opposite inclusion. This implies \( \ptm \circ \mtp = I \vert_{\uniqueminus } \).
	The proof for the corresponding statement in \ref{prop it:2} follows the same reasoning.
\end{proof}
We will sometimes write  \( \ptm^{-1} = \mtp \), to refer to the statement of the Proposition \ref{prop:analysis of ptm}. However, one should keep in mind that both operators are unbounded and their product is defined only on subdomains of \( \lts \).

In the following theorem we establish that the operators \( \ptm \) and \( \mtp \) are closed. This property is paramount for what follows.

\begin{thm}\label{thm:closed graph}
	The graph of the operator \( \ptm \) (resp. \( \mtp \)) is weakly closed in \( \lts\slash \scalp{1} \times  \lts \) (resp. \( \lts \times  \lts\slash \scalp{1}\)). In particular, \( \ptm \) and \( \mtp \) are closed operators.
\end{thm}

\begin{proof}
	Let \( \zeroset \Subset \SS \) be open. It is known that the space of tangent and divergence-free vector fields on $\Sigma$ is weakly closed.
	Now, let \( ( \varphi_{n}, \phi_{n} )_{n\in\bbN} \) be a sequence in the graph of \( \ptm \) that converges weakly to \(( f, h) \in  \lts\slash \scalp{1} \times \lts \). Then in particular, the sequence \(( \varphi_{n})_{n\in\bbN} \) in \( \uniqueplus  \) converges weakly to \( f \in \lts\slash \scalp{1} \), and the sequence \( (\phi_{n})_{n\in\bbN} \)  converges weakly to \( h \in \lts \). From Proposition \ref{prop:analysis of ptm} we know that \( \phi_{n} = \ptm(\varphi_{n})\) is in \( \uniqueminus  \) for each \( n \); and by Corollary \ref{cor:map between hardy} we know that for each \( n \) the field \( B_{+}\varphi_{n} + B_{-} \phi_{n} \) is tangent and divergence-free on \( \zeroset \). Moreover, since it holds for every \( \g \in \glts \) that
	\begin{align}
		|\scalp{B_{+} \varphi_{n} + B_{-} \phi_{n}, \g} | \leq |\scalp{\varphi_{n}, B_{+}^{\ast}\g}| + |\scalp{\phi_{n}, B_{-}^{\ast}\g}|,
	\end{align}
	the sequence \(  \big(B_{+}\varphi_{n} + B_{-} \phi_{n} \big)_{n\in\bbN}  \) converges weakly to some \( \f \) in \( \glts \). Due to the weak closure of the space of tangent and divergence-free vector fields, \( \f \) is tangent and divergence-free on \( \zeroset \). From Theorem \ref{thm:main} and Corollary \ref{cor:map between hardy}, there is a unique \( \varphi_{0} \in \uniqueplus  \) such that the Hardy components of \( \f \) are \( B_{+}\varphi_{0} \)  and \( B_{-} \big(\ptm(\varphi_{0}) \big) \).
	It follows that for every \( \g \in \glts \) we have
	\begin{align}\label{eq:weak convergence}
		|\scalp{B_{+}\big(\varphi_{n} - \varphi_{0}\big) + B_{-}\big(\ptm(\varphi_{n}) - \ptm(\varphi_{0})\big), \g}| \longrightarrow \ 0
		\qquad \text{as } n \to \infty.
	\end{align}
	In particular, using the fact that the Hardy spaces are orthogonal, we see
	\begin{align}
		|\scalp{B_{+} \big( \varphi_{n}  - \varphi_{0} \big), \g}|
		= |\scalp{ \varphi_{n}  - \varphi_{0}, B_{+}^{\ast} \g}| \longrightarrow 0
		\qquad \text{as } n \to \infty,
	\end{align}
	for every \( \g \in \Hp \).
	Since \( B_{+} \) is invertible, so is \( B_{+}^{\ast} \). Especially, \( B_{+}^{\ast} \) is onto \( \lts \slash \scalp{1} \), and thus, \( (\varphi_{n})_{n\in\bbN} \) converges weakly to \( \varphi_{0} \). By the uniqueness of limit points, we have \( \varphi_{0} = f \).
	Choosing \( \g \) in \eqref{eq:weak convergence} to be in \( \Hm \)  and using the same arguments, we see that \( \ptm(\varphi_{n}) \) converges weakly to \( \ptm(\varphi_{0}) = h \). Consequently, the graph of \( \ptm \) is weakly closed. Moreover, weakly closure and norm closure are equivalent for convex set in Hilbert space(e.g., \cite[Thm. 3.12]{rudin91}). Hence, \( \ptm \) is closed, and thus so is \( \mtp = \ptm^{-1} \).
\end{proof}

\begin{cor}\label{cor:alternative}
	Let \( \zeroset \Subset \SS \) be open. A function \( f \in \lts \) is in $\uniqueplus$ if and only if there exists a sequence of functions  $(f_n)_{n\in\bbN}\subset \uniqueplus$ such that $\lim_{n\to\infty}\|f_{n}- f\|=0$ and $\left(\|\ptm(f_n)\|\right)_{n\in\bbN}$ is bounded; in that case, $\ptm(f_{n})$ converges to $\ptm(f)$ in the weak sense.
\end{cor}
\begin{proof}
For the ``only if'' statement it is enough to choose the sequences of functions \( f_{n} = f \), which satisfies the desired properties.

For the ``if'' statement, let \( (f_{n})_{n\in\bbN} \) be a sequence in \( \uniqueplus \) that converges to \( f \in \lts \) and assume that $\left(\|\ptm(f_n)\|\right)_{n\in\bbN}$ is bounded. Recall, that a closed unit ball of a Hilbert space is compact in the weak topology (e.g.\ \cite[Thm. 1.6.7]{kadison1997fundamentalsI}), which implies that $\left(\ptm(f_n)\right)_{n\in\bbN}$ has weak cluster points.
Let $g$ be one of these cluster points and $(\ptm (f_{n_{i}}))_{i\in\bbN}$ be a subsequence that weakly converges to $g$. Consequently, \( (f_{n_{i}}, \ptm(f_{n_{i}})) \) is a sequence in the graph of \( \ptm \), weakly converging to \( (f, g) \). Since the graph of \( \ptm \) is weakly closed by Theorem \ref{thm:closed graph}, it follows that \( f \) is in $\uniqueplus $ and thus $g=\ptm (f)$ and $\ptm (f_n)$ weakly converges to $\ptm (f)$.
\end{proof}

\begin{cor}\label{cor:closednessoffeasibleregion}
	Let \( \zeroset \Subset \SS \) be open. Then for any constant $\Cone>0$, the set
	$$\uniqueplus^{\Cone}:=\{g\in \uniqueplus: \|\ptm(g)\|\leq \Cone\}$$
	is closed in $\lts$.
\end{cor}
\begin{proof}
	Let $(g_n)_{n\in\bbN}$ be a Cauchy sequence in \( \uniqueplus^{\Cone} \) that converges to $g$. Because $\|\ptm(g_n)\|$ is bounded, Corollary \ref{cor:alternative}  states that $g$ is in $\uniqueplus$  and that $\ptm(g_n)$ weakly converges to $\ptm(g)$. Thus, we have
	\begin{equation}
		\langle \ptm(g),\ptm(g)\rangle
		= \lim_{n\rightarrow \infty}\langle \ptm(g),\ptm(g_n)\rangle
		\leq \Cone \|\ptm(g)\|,
	\end{equation}
	which gives $\|\ptm(g)\|\leq \Cone$. Therefore $g$ is in $\uniqueplus^{\Cone}$ and $\uniqueplus^{\Cone}$ is a closed set.
\end{proof}

The closedness of $\ptm$ and $\mtp$ and the subsequent corollaries now allow us to discuss the initial bounded extremal problem \eqref{eqn:BEP1}. It yields that \eqref{eqn:BEP1} is well-defined for any $f$ in $\lts$ instead of just in the dense subspace $\uniqueplus$. This give us the freedom to consider noise contaminated input. Although this paper will not go into any reconstruction algorithms, we would still like to established the following properties.

\begin{cor}\label{cor:BEP1_P1}
	Let \( \zeroset \Subset \SS \) be open. For any $f\in \lts$ and any fixed $\Cone>0$, a minimizer of \eqref{eqn:BEP1} exists and this minimizer is unique. Moreover, if $\varphi_{\Cone}$ denotes the minimizer of \eqref{eqn:BEP1}, then for $f\in \lts/\langle1\rangle$, $\|\varphi_{\Cone}-f\|\rightarrow 0$ as $\Cone\rightarrow \infty$.
\end{cor}
\begin{proof}
	From Corollary \ref{cor:closednessoffeasibleregion}, we know that $\uniqueplus^{\Cone}$ is a closed set. Furthermore, the set $\uniqueplus^{\Cone}$ is convex by linearity of $\ptm$ and non-empty because it contains the zero function. Hence, a minimizer for \eqref{eqn:BEP1}, %
	as the best approximant of $f$ from $\uniqueplus^{\Cone}$, uniquely exists by the Hilbert projection theorem.
	
	Since the set $\uniqueplus^{\Cone}$ enlarges with increasing $\Cone$, the remainder $\|\varphi_{\Cone}-f\|$ must decrease correspondingly. In consequence, for $f\in \lts/\langle1\rangle$, the limit of $\|\varphi_{\Cone}-f\|$, as $\Cone\to\infty$, must be zero by the density of $\uniqueplus$ in $\lts/\langle1\rangle$.
\end{proof}

\begin{cor}\label{cor:BEP1_P2}
	Let $\varphi\in \uniqueplus$ and $\Cone>0$ be a constant such that $\|\ptm(\varphi)\|\leq \Cone$. Furthermore, let $(f_n)_{n\in\bbN}\subset \lts$ be a sequence that converges to $\varphi$, and let $\varphi_{n,\Cone}$ denote the minimizer of \eqref{eqn:BEP1} corresponding to the choice $f=f_n$. Then, $\left(\ptm(\varphi_{n,\Cone})\right)_{n\in\bbN}$ weakly converges to $\ptm(\varphi)$ as $n\to\infty$.
\end{cor}
\begin{proof}
	We first notice that $\varphi$ is in $\uniqueplus^{\Cone}$. Furthermore, by definition of \eqref{eqn:BEP1}, it holds that $\|\varphi_{n,\Cone}-f_n\|\leq\|\varphi -f_n\|$. Thus, $\|\varphi_{n,\Cone}-\varphi\|\leq2\|\varphi-f_n\|$, which means that $\left(\varphi_{n,\Cone}\right)_{n\in\bbN}$ converges to $\varphi$ as $n\to\infty$. Since, by construction, $\|\ptm(\varphi_{n,\Cone})\|$ is bounded by $\Cone$, Corollary \ref{cor:alternative}  yields that $\ptm(\varphi_{n,\Cone})$ weakly converges to $\ptm(\varphi)$ as $n\to\infty$.
\end{proof}

\begin{rem} \label{rem:equvalentDefinition}
	We could have equivalently formulated \eqref{eqn:BEP1} in the following way:  Find $(\varphi,\psi)$ in $G(\ptm)$, where $G(\ptm)$ denotes the graph of $\ptm$, such that $\| P_2(\varphi,\psi)\|\leq \Cone$ and
	\begin{equation*}\label{eqn:BEP1b}
		\begin{aligned}
			 \| P_1(\varphi,\psi)- f\|=\inf\left\{\| P_1(g,h)- f\|:\,(g,h)\in G(\ptm),\,\| P_2(g,h)\|\leq \Cone\right\}.
		\end{aligned}
	\end{equation*}
	By $P_1$ and $P_2$ we denote the canonical projections acting via $P_1(g,h)=g$ and $P_2(g,h)=h$, respectively. Since we know that $G(\ptm)$ is a closed subspace of the Hilbert space $\lts \times \lts$ and $P_1$ and $P_2$ are coprime in the sense of \cite{Chalendar2003}, the two previous corollaries would also follow by similar arguments used for other BEPs. %
\end{rem}

Corollary \ref{cor:BEP1_P2} implies that if we have a bound on $\|\ptm(\varphi)\|$ and an approximating sequence $(f_n)_{n\in\bbN}$ for $\varphi$%
, we can hope to reconstruct $\ptm(\varphi)$ in the weak sense to an acceptable precision. Nonetheless, because $\ptm$ is unbounded, an upper bound on $\|\ptm(\varphi)\|$ reflects crucial prior information required for the solution of the problem. We can further analyze the weak convergence and the influence of the upper bound with the help of the adjoint $\ptm^{\ast}$ and a related bounded extremal problem as we show in Appendix \ref{app:opsextremalprobs}. This setup requires the explicit evaluation of $\ptm$ and possibly $\ptm^*$. However, the specific structure of these operators allows us to formulate another bounded extremal problem that does not require their explicit evaluation; in particular, no evaluation of the operators $[P(K\pm\tfrac{1}{2})]^{-1}$ is required. This is discussed in the next section and addresses the quantification of weak convergence slightly different from Appendix \ref{app:opsextremalprobs}.

\subsection{Approximation of the modes via a second bounded extremal problem}\label{sec:bep2}

Inspired by the weak convergence features of \eqref{eqn:BEP1}, in this section, we consider a fairly related problem of approximating certain modes of $\ptm(\varphi)$, namely, $\langle \ptm(\varphi),\dualelement\rangle$ with $\dualelement$ being fixed, e.g., a spherical harmonic. This leads to the problem of finding an adequate function $h$ in $\lts$ such that
\begin{align}
	\langle \ptm(\varphi),\dualelement\rangle=\langle \varphi,h\rangle,\qquad\textnormal{for all }\varphi\in\uniqueplus.\label{eqn:goalmode}
\end{align}
Equality in \eqref{eqn:goalmode} is unlikely to be achieved in general, unless $\dualelement$ is in $\dom{\ptm^*}$, in which case one can choose $h=\ptm^*(\dualelement)$. Since, however, $\dom{\ptm^*}$ is dense but not closed in $\lts$ we cannot guarantee that our choice of $\dualelement$ will satisfy this condition. What we can hope for is an approximate solution for \eqref{eqn:goalmode}, which is linked to the method of approximate inverse from \cite{Louis1996,LouMas90} that has been applied to related inverse magnetization problems in a modified form, e.g., in \cite{barchehar18,baratchartgerhards16}. The particular structure of $\ptm$ and $\ptm^*$ leads us to consider the following bounded extremal problem.

\paragraph{Bounded Extremal Problem 2.} Let $E$ denote the projection from $\lts$ onto $\divfreezeroset$. Given a $\dualelement$ in $\lts$ and a fixed $\Ctwo >0$, find an $h$ in $L^2(\zeroset)$ with $\|h\|\leq \Ctwo$ such that
\begin{equation}\label{eqn:BEP3}\tag{$BEP_{2}$}
	\begin{aligned}
		\| E\dualelement- E(K+\tfrac{1}{2} \id)Ph\|=\inf\left\{\| E\dualelement- E(K+\tfrac{1}{2} \id)Pg\|:\,g\in L^2(\zeroset),\,\|g\|\leq \Ctwo\right\}.
	\end{aligned}
\end{equation}

We know that $E\left(K + \tfrac{1}{2} \id \right)\proj$ is a bounded linear operator from $\lts$ into $\lts$, and its adjoint is given by $\proj \left(K + \tfrac{1}{2} \id \right)E$. The latter operator is, by Theorem \ref{prop:dense1}, injective from  $\divfreezeroset$ into $L^2(\zeroset)$ with dense range. Then we can characterize the solution of \eqref{eqn:BEP3} following the statement of constrained approximation in Hilbert spaces from \cite[Lemma 2.1, Theorem 2.1]{Chalendar2003}.%
\begin{prop}\label{cor:BEP3_P1}
	Let \( \zeroset \Subset \SS \) be open. A minimizer of \eqref{eqn:BEP3} exists and is unique. Moreover, let $h_{\Ctwo}$ be the solution of \eqref{eqn:BEP3}, and set $M=E\left(K + \tfrac{1}{2} \id \right)\proj$. The constraint is saturated, i.e. $\|h_{\Ctwo}\|= \Ctwo$, if $E\dualelement$ is not in the range of $M$. In this case, $h_{\Ctwo}$ satisfies
	\begin{equation}\label{eqn:solutionBEP3}
		(M^{\ast}M-\gamma \id)h_{\Ctwo}=M^{\ast}(E\dualelement),
	\end{equation}
	where $\gamma<0$ is the unique constant such that $\|h_{\Ctwo}\|= \Ctwo$ .
\end{prop}

We can now state the main theorem of this section, quantifying the approximation error for $\langle \ptm(\varphi),\dualelement\rangle$ by using  \eqref{eqn:BEP3}.

\begin{thm}
	Let \( \zeroset \Subset \SS \) be open and $\dualelement\in\lts$ and $\Ctwo>0$ be fixed. Furthermore, let $h_{\Ctwo}$ be the solution to \eqref{eqn:BEP3} and define $L_{\dualelement}(\Ctwo)=\|E\dualelement-E\left(K + \tfrac{1}{2} \id \right)\proj h_{\Ctwo}\|$. Then, it holds for any $\varphi \in \uniqueplus$ and any noisy version $\varphi_\eps\in\lts$, with $\|\varphi-\varphi_\eps\|<\eps$, that
	\begin{equation} \label{eqn:errorestimationBEP3_2}
		\left|\langle \ptm(\varphi), \dualelement\rangle -\left(\langle \proj\varphi_{\varepsilon}, h_{\Ctwo}\rangle-\langle  \varphi_{\varepsilon},\dualelement\rangle\right)\right|\leq (\|\varphi\|+\|\ptm(\varphi)\|)L_{\dualelement}(\Ctwo)+\varepsilon(\Ctwo+\|\dualelement\|),
	\end{equation}
	where $L_{\dualelement}(\Ctwo)$ decreases monotonically to zero as $\Ctwo\to\infty$.
\end{thm}

\begin{proof}
	The monotonic decrease of $L_{\dualelement}(\Ctwo)$ to zero follows from the density of the range of $E \left(K + \tfrac{1}{2} \id \right)P$ in $\divfreezeroset$.

Since $\varphi$ is in $\uniqueplus$, there is a unique function $g=Eg$ in $\divfreezeroset$ such that $\proj \varphi = \proj(K + \tfrac{1}{2} \id) g$ and $\ptm(\varphi)=g-\varphi$, which yields
	\begin{align}\label{eq: ptm as difference}
		\langle \ptm(\varphi), \dualelement\rangle&=\langle g, E\dualelement\rangle-\langle \varphi, \dualelement\rangle.
	\end{align}
	Furthermore, we get
	\begin{align}\label{eq: weak equality for E}
		\langle g, E\dualelement\rangle&=\langle g, E\dualelement-E\left(K + \tfrac{1}{2} \id \right)\proj h_{\Ctwo}\rangle+\langle g, E\left(K + \tfrac{1}{2} \id \right)\proj h_{\Ctwo}\rangle\nonumber
		\\&=\langle g, E\dualelement-E\left(K + \tfrac{1}{2} \id \right)\proj h_{\Ctwo}\rangle+\langle \proj\left(K + \tfrac{1}{2} \id \right)g, h_{\Ctwo}\rangle
		\\&=\langle g, E\dualelement-E\left(K + \tfrac{1}{2} \id \right)\proj h_{\Ctwo}\rangle+\langle \proj\varphi, h_{\Ctwo}\rangle\nonumber
	\end{align}
	Combining \eqref{eq: ptm as difference} and \eqref{eq: weak equality for E}, we obtain for the remainder $\langle \ptm(\varphi),\dualelement\rangle -\left(\langle \proj\varphi_{\varepsilon}, h_{\Ctwo}\rangle-\langle  \varphi_{\varepsilon},\dualelement\rangle\right)$,
	\begin{align}
		\langle \ptm(\varphi), \dualelement\rangle &-\left(\langle \proj\varphi_{\varepsilon}, h_{\Ctwo}\rangle-\langle  \varphi_{\varepsilon},\dualelement\rangle\right) = \langle g, E\dualelement\rangle-\langle \varphi, \dualelement\rangle-\left(\langle \proj\varphi_{\varepsilon}, h_{\Ctwo}\rangle-\langle  \varphi_{\varepsilon},\dualelement\rangle\right)\nonumber
		\\&=\langle g, E\dualelement-E\left(K + \tfrac{1}{2} \id \right)\proj h_{\Ctwo}\rangle +\langle \proj(\varphi- \varphi_\eps), h_{\Ctwo}\rangle +\langle  	\varphi_{\varepsilon}- \varphi, \dualelement\rangle.\label{eqn:tptmvarphi}
	\end{align}
	The last term in \eqref{eqn:tptmvarphi} is bounded by $\eps\|\dualelement\|$, the middle term by $\eps \Ctwo$, and the first term by $\|g\|L_{\dualelement}(\Ctwo)$. Since $\|g\|=\|\varphi+\ptm(\varphi)\|\leq \|\varphi\|+\|\ptm(\varphi)\|$, the desired estimate follows.
\end{proof}

\begin{rem} \label{rem:ontwoBEPs}
	Although \eqref{eqn:BEP1} and \eqref{eqn:BEP3} both relate to the problem of reconstructing $\ptm(\varphi)$ in the weak sense, they might be suitable for different situations. For example, \eqref{eqn:BEP1} is suitable if one can find a good set of basis functions $\varphi_{n,c}$ in $\uniqueplus$ for which one is able to compute $\ptm(\varphi_{n,c})$. Opposed to this, \eqref{eqn:BEP3} does not require the possibly difficult evaluation of $\ptm$ and it is independent of the actual data. But it only provides information on single modes of $\ptm(\varphi)$.
\end{rem}

\bibliography{biblio}
\bibliographystyle{plain}

\paragraph{Acknowledgments.} The authors have been partially funded by BMWi (Bundesministerium f\"ur Wirtschaft und Energie) within the joint project 'SYSEXPL -- Systematische Exploration', grant ref. 03EE4002B.

\appendix

\section{Appendix}
\subsection{Hardy spaces}\label{sec:hardyspace}
We briefly recapitulate the classical definition of Hardy spaces as mentioned, e.g., in \cite{stein1960}. Letting $\textnormal{Harm}(\BB)=\{u\in C^\infty(\BB): \Delta u=0 \textnormal{ in }\BB\}$ and $\textnormal{Harm}(\R^d\setminus\overline{\BB})=\{u\in C^\infty(\R^d\setminus\overline{\BB}): \Delta u=0 \textnormal{ in }\R^d\setminus\overline{\BB},\,\lim_{|x|\to\infty}|u(x)|=0\}$ be the spaces of harmonic functions on the corresponding domains, we set
\begin{align*}
H_+(\BB)&=\left\{\grad u:u\in \textnormal{Harm}(\BB),\,\sup_{r\in[0,1)}\int_{\SS_r}|\grad u(y)|^2d\omega(y)<\infty\right\},
\\H_-(\R^d\setminus\overline{\BB})&=\left\{\grad v:v\in \textnormal{Harm}(\R^d\setminus\overline{\BB}),\,\sup_{r\in(1,\infty)}\int_{\SS_r}|\grad v(y)|^2d\omega(y)<\infty\right\}.
\end{align*}
The nontangential limits to the sphere of any vector fields in $H_+(\BB)$ or $H_-(\R^d\setminus\overline{\BB})$ are known to lie in  $L^2(\SS,\R^d)$. The space of boundary fields obtained by taking the non-tangential limits define the Hardy spaces on the sphere, such that
\begin{align*}
H_+(\SS)&=\{\f \in \glts \colon \ \f (y) = \operatorname*{nt-lim}_{\BB\ni x \to y} \nabla u (x) \text{ at almost every \( y \in \SS \)}, \nabla u \in H_+(\BB)\},
\\H_-(\SS)&=\{\g \in \glts \colon \ \g(y) = \operatorname*{nt-lim}_{\R^d\setminus\overline{\BB}\ni x \to y} \nabla u (x)  \text{ at almost every \( y \in \SS \)}, \nabla u  \in H_-(\R^d\setminus\overline{\BB})\},
\end{align*}
where \( \operatorname*{nt-lim} \) denotes the non-tangential limits to the sphere.
It has been shown in \cite{bargerkeg20} that the definition of Hardy spaces that we use in this paper (via the operators $B_+$ and $B_-$ from \eqref{cross_s}) are equivalent to the standard definition of Hardy spaces as presented above.

\subsection{Extension of locally divergence-free fields}\label{sec:cont}

In the following, if \( \zeroset \) is open, then \( \sob \) and \( \sobzero \) will denote the Sobolev spaces on \( \zeroset \) as defined in Section \ref{sec:intro}. Similarly, for the open set \( \zeroset \) we will write \( \sobspace{1,2}{\zeroset} \) to refer to the Sobolev space of degree one in the \( \zeroset \). Then, \( \sobpartialzero \) will denote the Sobolev space on the boundary of \( \zeroset \), and \( \dsobpartialzerohalf \) will be the dual space of \( \sobpartialzerohalf \). For  \( \varphi \in \dsobpartialzerohalf \) and \( \phi \in \sobpartialzerohalf \) we will denote their dual paring by \( \dualp{ \varphi, \phi} \). For better clarity, we will put the index on the dual paring and the scalar product to indicate the domain of functions involved in that operation.
In this section, the mapping \( \trace \colon \sobspace{1,2}{\zeroset} \to \sobpartialzerohalf\) will denote the trace from Sobolev functions on \( \zeroset \) to Sobolev functions on \( \partial \zeroset \). If \( \zeroset \) is Lipschitz, the trace between the indicated spaces is continuous and onto. Moreover, it has a continuous right inverse \( E \colon \sobpartialzerohalf \to \sobspace{1,2}{\zeroset}\), such that for every \( \varphi \in \sobpartialzerohalf \) we have \( \trace \circ E (\varphi) = \varphi \).

\begin{lem}\label{lem:continuationhelmholtz}
	Let \( \zeroset\subset \SS \) be a Lipschitz domain with connected boundary \( \partial \zeroset \). Then, for every \( \f \in \glts \) with \( \divs(\f)\vert_{\zeroset} = 0 \), there exists a \( \g \in \Divfree \) with \( \g = \f \) on \( \zeroset\).
\end{lem}

\begin{proof}
Since \( \f \) is divergence-free on \( \zeroset \), it has a unique normal \( \nn \cdot \f \) on \( \partial \zeroset \) as a distribution in %
\( \dsobpartialzerohalf \) defined by the Green's formula (similar to \cite[Thm. 4.4]{mclean2000strongly}),
\begin{align}\label{eqn:a5}
	\dualp{\nn \cdot \f , \varphi}_{\partial \zeroset} = \scalp{ \f , \grads (E\varphi)}_{\zeroset}
	\qquad \big(\varphi \in \sobpartialzerohalf\big).
\end{align}
By $\nn(x)$ we denote the unit vector field pointing outward of $\Sigma$, and that is tangential to the sphere $\SS$ and normal to the boundary $\partial\zeroset$ at almost every point $x\in\partial\zeroset$. The distribution \( \nn \cdot \f \) has vanishing mean since \( ( \nn \cdot \f, 1 ) = 0 \). Since \( \SS \setminus \overline{\zeroset} \) is also a connected Lipschitz domain if \( \zeroset \) is a Lipschitz domain with connected boundary, the Lax-Milgram lemma then guarantees the existence of a surface harmonic function $h$ in $W^{1,2}(\SS \setminus \overline{\zeroset})/\langle1\rangle$ such that
\begin{align}\label{eqn:a6}
	\langle \grads h,\grads \varphi\rangle_{\SS \setminus \overline{\zeroset}}=	\dualp{-\nn \cdot \f , \trace(\varphi)}_{\partial \zeroset} \qquad \big(\varphi \in W^{1,2}(\SS \setminus \overline{\zeroset})/\langle1\rangle\big),
\end{align}
i.e. 
whose Neumann boundary values $-\nn\cdot\grads h$ coincide with $-\nn\cdot\f$ on $\partial\zeroset$ (understood in the \( \dsobpartialzerohalf \)-sense since \( ( \nn \cdot \f, 1 ) = 0 \); the right hand side of \eqref{eqn:a6} is defined via \eqref{eqn:a5} as the trace \( \trace(\varphi) \) is a function in \( \sobpartialzerohalf \); the minus sign before $\nn$ appears because the outward-pointing normal with respect to \( \SS \setminus \zeroset \) is \( -\nn \)). %
Now, we set
\begin{align}
    \g =
	\begin{cases}
		\f & \text{on } \zeroset \\
		\grads h & \text{on } \SS \setminus \overline{\zeroset}.
\end{cases}
\end{align}
Then, \( \g \) is a square integrable tangent vector field on \( \SS \), and it is divergence-free on the whole sphere since for every test function \( \phi \in C^{\infty}_{0} (\SS) \) we have by the definition of the normal trace
\begin{align*}
    \scalp{\g, \grads \phi}_{\SS}
	= \scalp{ \f , \grads \phi}_{\zeroset} + \scalp{ \grads h , \grads \phi}_{\SS\setminus \overline{\zeroset}}
	&= \dualp{ \nn \cdot \f , \trace(\phi) }_{\partial \zeroset} - \dualp{\nn \cdot \grads h, \trace (\phi)}_{\partial \zeroset}
	= 0.
\end{align*}
Thus, \( \g \) is the desired globally divergence-free vector field.
\end{proof}

\subsection{Proof of Lemma \ref{eq:orthogonal complement}}\label{sec:prflemma}

\begin{proof}%
Let \( \zeroset \subset \SS \) be open with \( \SS\setminus \zeroset\) having positive Lebesgue measure. First, we observe that we can phrase \( \divfreezeroset \) as the orthogonal complement of the set
\begin{align}\label{eq:def G}
    G
	= \big\{
		(\grads S)^{\ast} \grads \varphi \colon \varphi \in \cc(\zeroset)
		\big\} \subset \lts.
\end{align}
This is because for every \( f \in \lts \)  we have
\begin{align}\label{eq:Q as complement}
		\divs(\grads S f)(\varphi)
		= \scalp{\grads S f, \grads \varphi}
		= \scalp{ f, (\grads S)^{\ast} \grads \varphi }
	\qquad \big( \varphi \in \cc(\zeroset) \big).
\end{align}
Thus, \( f \) is in \( \divfreezeroset \) (left-hand side of \eqref{eq:Q as complement} vanishes) if and only if \( f \) is orthogonal to \( G \) (right-hand side of \eqref{eq:Q as complement} vanishes). In particular, \( G \) is orthogonal to constant functions and thus, it is a subset of \( \lts \slash \scalp{1} \).
The orthogonal complement of \( \divfreezeroset \) is then the \( L^{2} \)-closure of \( G \),
\begin{align}
    \divfreezeroset^{\perp} = (G^{\perp})^{\perp} = \overline{G}.
\end{align}
Moreover, inserting the identity \( SS^{-1} = I \colon \sob \to \sob \) in \eqref{eq:def G} and using \eqref{eq:single-double layer} we can rewrite \( G \) as
\begin{align}\label{eq:set G}
    G
		= \big\{
				(\grads S)^{\ast} (\grads S) S^{-1} \varphi
				\colon \varphi \in \cc(\zeroset)
			\big\}
		= \big\{
				\big( K + \tfrac{1}{2}I \big)\big( K - \tfrac{1}{2} I \big)
				S^{-1}\varphi
				\colon \varphi \in \cc(\zeroset)
			\big\}.
\end{align}
To prove the assertion of the lemma, it remains to show that the closure of \( G \) is given by
\begin{align}
    H
		= \big\{
			\big( K + \tfrac{1}{2}I \big) \big( K - \tfrac{1}{2}I \big)
			S^{-1} \varphi \colon \varphi \in \sobzero
			\big\}.
\end{align}
For this assume a Cauchy sequence  \( (f_{n})_{n\in\bbN} \) in \( G \) that converges to \( f \in \lts \); for each \( f_{n} \) let \( \varphi_{n} \in \cc(\zeroset) \) be the corresponding smooth function as in \eqref{eq:set G} with the mean \( c_{n} = \scalp{\varphi_{n},1} \).
  Since \( (f_{n})_{n\in\bbN} \) is Cauchy and the operators \( K \pm \tfrac{1}{2}I \) are both isomorphisms on \( \lts \slash \scalp{1} \), the sequence of functions \( S^{-1}\varphi_{n} \) is Cauchy in the space \( \lts  \).
Correspondingly, the sequence of functions \( S^{-1}\varphi_{n} - S^{-1}c_{n} \) converges in \( \lts \slash \scalp{1} \) (since \( S^{-1}c_{n} \) is the mean of \( S^{-1}\varphi_{n} \) due to the fact that \( S \) preserves constant functions).
Because \( S \colon \lts \to \sob \) is bounded, it further follows that the sequence \( (\varphi_{n} - c_{n} )_{n\in\bbN}\) must be Cauchy in \( \sob \).
In particular, it converges in the \( L^{2} \)-sense on every subdomain of \( \SS \) that has positive measure --- especially on \( \SS \setminus\zeroset \), where it is a sequence of constant functions \( c_{n} \) (since each \( \varphi_{n} \) is compactly supported on \( \zeroset \)). Clearly, \( (c_{n})_{n\in\bbN} \) must then also converge on the entire sphere in the \( L^{2} \)-sense and in the \( W^{1,2} \)-sense. It follows that both sequences, \( (\varphi_{n} - c_{n} )_{n\in\bbN}\)  and \( (c_{n})_{n\in\bbN} \), are individually Cauchy in \( \sob \). Correspondingly, the sequence of functions \( \varphi_{n} = (\varphi_{n}-c_{n}) + c_{n}\) converges in \( \sob \) and its limit is a function \( \varphi_{0} \) in \( \sobzero \) (since by  definition \( \sobzero \) is the closure of \( \cc(\zeroset) \) in the Sobolev norm).
Moreover, we have
\(	\| \big( K + \tfrac{1}{2}I \big) \big( K - \tfrac{1}{2}I \big) S^{-1} \varphi_{0} - f_{n} \| \to 0\) for \( n \to \infty \), and thus,
\begin{align}
        f = \big( K + \tfrac{1}{2}I \big) \big( K - \tfrac{1}{2}I \big) S^{-1} \varphi_{0},
\end{align}
implying that \( \overline{G} \subset H \). The inverse inclusion is trivial, since all involved operators are isomorphic on the corresponding spaces. Thus, \( \overline{G} = H \) which concludes the proof.
\end{proof}

\subsection{Properties of $\ptm^*$, $\mtp^*$, and a quantification of weak convergence}\label{app:opsextremalprobs}
From the standard theory of operators we know that
the adjoint operators of densely defined operator are always closed (e.g., \cite[thm. II.2.6]{goldberg2006unbounded}).
Since the operators \( \ptm \) and \( \mtp \) are densely defined, closed, injective and have a dense range, their adjoints have the same properties.
Moreover, also for unbounded operators we have \( (A^{\ast})^{-1} = (A^{-1})^{\ast} \) (if the unbounded operator \( A \) has an (unbounded) inverse \( A^{-1} \)). Therefore, by Proposition \ref{prop:analysis of ptm} we have
\begin{align}
	\mtp^{\ast}
	= (\ptm^{-1})^{\ast}
	=  (\ptm^{\ast})^{-1}
	&&
	\text{and}
	&&
	\ptm^{\ast}
	= (\mtp^{-1})^{\ast}
	=  (\mtp^{\ast})^{-1}.
\end{align}

We can now formulate the following bounded extremal problem.

\paragraph{Bounded Extremal Problem 3.} \label{BEP2}Let $\dualelement\in \lts$ and $\Ctwo>0$ be fixed. Find a $\psi$ in the domain $\textnormal{dom}(\ptm^{\ast})$ of the adjoint operator with $\|\ptm^*(\psi)\|\leq \Ctwo$ such that
\begin{equation}\label{eqn:BEP2}\tag{$BEP_{3}$}
	\begin{aligned}
		\| \psi- \dualelement\|=\inf\left\{\| g- \dualelement\|:\,g\in \textnormal{dom}(\ptm^{\ast}),\,\|\ptm^{\ast}(g)\|\leq \Ctwo\right\}.
	\end{aligned}
\end{equation}

Because $\ptm^{\ast}$ is weakly closed with dense domain and dense range, we can easily establish a parallel result to Corollary \ref{cor:BEP1_P1}, namely, that the minimizer of \eqref{eqn:BEP2} exists. This leads us to the following estimate on the weak convergence in Corollary \ref{cor:BEP1_P2}.

\begin{prop}
	Let the setup be as in Corollary \ref{cor:BEP1_P2} and let $\varphi_{n,\Cone}$ denote the minimizer of the corresponding \eqref{eqn:BEP1}. Furthermore, let $\psi_{\Ctwo}$ denote the minimizer of \eqref{eqn:BEP2} for a given $\Ctwo>0$ and and define $J_{\dualelement}(\Ctwo)=\|\dualelement-\psi_{\Ctwo}\|$. Then, $J_{\dualelement}(\Ctwo)$ decreases monotonically to zero as $\Ctwo \to \infty$, and for any $\dualelement\in \lts$ we have
	\begin{equation}
		\left|\langle \ptm(\varphi_{n,c}), \dualelement\rangle-\langle \ptm(\varphi), \dualelement\rangle\right|\leq 2 \min_{\Ctwo>0}\left\{\Cone J_{\dualelement}(\Ctwo) +\Ctwo \|f_n-\varphi\| \right\}.\label{eqn:lastres}
	\end{equation}
\end{prop}

\begin{proof}
We can directly estimate
\begin{align}
	&\left|\langle \ptm(\varphi_{n,\Cone}), \dualelement\rangle-\langle \ptm(\varphi), \dualelement\rangle\right| \nonumber
	\\&=\left|\langle\nonumber \ptm(\varphi_{n,\Cone})- \ptm(\varphi), \dualelement\rangle\right|\\
	&\leq\left|\langle \ptm(\varphi_{n,\Cone})- \ptm(\varphi), \dualelement- \psi_t\rangle\right|+\left|\langle \ptm(\varphi_{n,\Cone})- \ptm(\varphi), \psi_t\rangle\right|\\ \nonumber
	&\leq\|\ptm(\varphi_{n,\Cone})- \ptm(\varphi)\|\cdot\|\dualelement- \psi_t\|+\left|\langle \varphi_{n,\Cone}- \varphi, \ptm^{\ast}(\psi_t)\rangle\right| \\ \nonumber
	&\leq 2\Cone J_{\dualelement}(\Ctwo) +2\|f_n-\varphi\| \Ctwo.
\end{align}
Since the above holds for any $\Ctwo>0$, we end up with \eqref{eqn:lastres}. The monotonic decrease of $J_{\dualelement}(\Ctwo)$ to zero is again a direct consequence of the density of the domain $\textnormal{dom}(\ptm^{\ast})$ in $\lts$.
\end{proof}

\end{document}